\documentclass{article}

\usepackage{amssymb,amsfonts,amsthm,amsmath}
\usepackage{mathrsfs}
\usepackage{authblk}
\usepackage[utf8]{inputenc}
\usepackage[english]{babel}

\newtheorem{theorem}{Theorem}
\newtheorem{lemma}{Lemma}
\newtheorem{corollary}{Corollary}
\newtheorem{thmletter}{Theorem}

\theoremstyle{definition}

\newtheorem{remark}{Remark}

\newcommand\blfootnote[1]{%
	\begingroup
	\renewcommand\thefootnote{}\footnote{#1}%
	\addtocounter{footnote}{-1}%
	\endgroup
}

\newcommand\R{\mathbb{R}}

\numberwithin{equation}{section}

\title{Approximation of Hausdorff operators}
\date{}

\author{A. Debernardi\footnote{Corresponding author. E-mail: \texttt{adebernardipinos@gmail.com}} }
\author{E. Liflyand}

\affil{\footnotesize Department of Mathematics, Bar-Ilan University, 52900 Ramat-Gan, Israel}

\begin{document}
	
	\maketitle
	
\begin{flushleft}
	\small{AMS 2010 subject classification: 41A25, 42A38, 44A15.}
\end{flushleft}

	\section*{Abstract}
Truncating the Fourier transform averaged by means of a generalized Hausdorff operator,
we approximate the adjoint to that Hausdorff operator of the given function. We find the formulas for the rate of approximation in various metrics in terms of the parameter of truncation and the components of the Hausdorff operator. Explicit rates of approximation and comparison with approximate identities are given in the case of Lipschitz $\alpha$ continuous functions. As an application, we show not only how to approximate the Hausdorff operator of a function, but the function itself in the $L^\infty$ norm. After numerous works on the boundedness of Hausdorff operators on various function spaces, this paper is the first application of Hausdorff operators to the problem of constructive approximation.

\blfootnote{The first author was supported by the ERC starting grant No. 713927 and the ISF grant No. 447/16.}

\section{Introduction}

The classical Hausdorff operator is defined, by means of a kernel $\varphi$, as

\begin{eqnarray}\label{haus1}
({\mathcal H}_{\varphi}f)(x)=\int_{\mathbb R}\frac{\varphi(t)}{|t|}f\Big(\frac{x}t\Big)\,dt,
\end{eqnarray}
and, as is shown first in \cite{Ge} (see also \cite{LM1} or \cite{emj}), such an operator is bounded in $L^1(\R)$ whenever $\varphi\in L^1(\R)$.

In the last two decades various problems related to Hausdorff operators attracted much attention.
The number of publications is growing considerably; to add some most notable, we mention \cite{AL, K, LL, LiMi, LiMi1, Mir}. There are two survey papers: \cite{CFW} and \cite{emj}). In the latter, as well as in \cite{Lif}, numerous open problems are given.

The Hausdorff operator \eqref{haus1} is expected to have better Fourier analytic properties than $f$. For example, in general, the inversion formula
\[
f(x)=\frac{1}{2\pi}\int_\R \widehat{f}(y) e^{ixy}\, dy
\]
does not hold for $f\in L^1(\R)$; in order to ``repair" this, one may consider some transformation of the function $f$ 
or its Fourier transform. In the case of the Hausdorff operator, what we expect is that
\begin{equation}
\label{EQtransforms}
\int_\R (\mathcal{H}_\varphi f)^{\widehat{}}(y)e^{ixy}\,dy \qquad \text{or}\qquad \int_\R (\mathcal{H}_\varphi\widehat{f})(y)e^{ixy}\,dy
\end{equation}
characterizes $f$, in a sense. The latter way seems to be more natural.

Here we analyze not this Hausdorff operator but a more general one, apparently first considered in \cite{Kuang} (see also \cite{Kuang1}). 
Given an odd function $a$ such that $|a(t)|$ is decreasing, positive, and bijective on $(0,\infty)$ 
(so that both $|a|$ and $1/|a|$ possess inverse functions in such an interval), we define
\begin{equation}
\label{EQgeneraldefinition}
({\mathcal H}f)(x)=({\mathcal H}_{\varphi,a}f)(x)=\int_{\mathbb R}{\varphi(t)}|a(t)|f(a(t)x)\,dt\,.
\end{equation}
It is clear that \eqref{haus1} corresponds to \eqref{EQgeneraldefinition} with $a(t)=t^{-1}$, and one can easily derive the corresponding results from the general ones. Moreover, we consider some such particular cases as examples. There is one more reason for considering the most general case: it is a proper basis for future
multidimensional extensions. Indeed, the most general form of the Hausdorff operator in several dimensions is, for $x\in \R^n$,
\begin{equation}
\label{EQhausmultidim}
\mathcal{H}f(x)=\mathcal{H}_{\varphi,A}f(x)=\int_{\R^n} \varphi(u) f(xA(u))\,du,
\end{equation}
where $A(u)$ is an $n\times n$ matrix with entries  $\{a_{i,j}(u)\}_{i,j=1}^n$, depending on $u$. The matrix $A(u)$ should be nonsingular almost everywhere, and $xA(u)$ denotes the usual product of a vector and a matrix. This form was introduced independently in \cite{BM} and \cite{LL}. It is clear that \eqref{haus1} is the particular case of \eqref{EQhausmultidim} in dimension one. Considering the multivariate Hausdorff operator to be of the form \eqref{EQhausmultidim} proved to be most beneficial. To mention some achievements, we refer to \cite{LL}, where the problem of the boundedness of the Hausdorff operator in the real Hardy space was solved, and \cite{LiMi1}, where not only the results from \cite{LiMi} were extended, but new phenomena related to the Hardy spaces $H^p$, with $0<p<1$, were discovered.

The consideration of these ``alternative'' transformations such as \eqref{EQtransforms} requires developing a parallel theory to Fourier integrals. In this paper we address two basic issues of approximation theory applied to (generalized) Hausdorff operators.
\begin{enumerate}
	\item To find the operator $T$ such that the integrals  of the type
	\[
	\int_{-N}^N ({\mathcal H}_{\varphi,a}\widehat{f})(y)e^{ixy}\,dy
	\]
	approximate $Tf$ as $N\to \infty$, for reasonable choices of $\varphi$. As we will see, the operator $T$ 
is by no means the identity operator. It is nothing but $\mathcal{H}^*$, the dual operator of $\mathcal{H}$, formally defined by the relation
	\[
	\int_{\mathbb R} \mathcal{H} f(x) g(x) \,dx=\int_{\mathbb R} f(x) \mathcal{H}^*g(x)\,dx.
	\]
	\item To study the rate of convergence to $\mathcal{H}^*f$ of the partial integrals
	\[
	\int_{-N}^N ({\mathcal H}_{\varphi,a}\widehat{f})(y)e^{ixy}\,dy,
	\]
	as $N\to \infty$ in the $L^p$ norm, where $1\leq p\leq \infty$.
\end{enumerate}
In particular, the problem of exploiting Hausdorff operators in approximation is risen. Indeed, application of analytic results in approximation seems to be the most convincing proof of their usefulness. The present work is the first attempt to understand what kind of approximation problems may appear in the theory of Hausdorff operators and solve some of them. The obtained results will open new lines in both the theory of Hausdorff operators itself and approximation theory. The difference between Hausdorff means and more typical multiplier (convolution) means, which comes from the difference between dilation invariance for the former and shift invariance for the latter, leads not only to new results but also to novelties in the methods.

The structure of the paper is as follows. In the next section, starting with certain preliminaries, we then
formulate the main results. In the section following after that, we prove the main results. Section~\ref{SECexamples} is devoted to present several examples of operators and their approximation estimates. Since after several works on the boundedness of the Hausdorff operators on various function spaces, this paper is the first application of Hausdorff operators to the problems of constructive approximation, in particular we compare the obtained results with their traditional counterparts (approximate identities given by convolution type operators). Finally, in Section~\ref{SEC5} we give concluding remarks, including how to approximate a function by means of Hausdorff operators in the $L^\infty$ norm.

Throughout the paper we denote, for $1\leq p\leq \infty$,
\[
\omega(f;\delta)_p=\sup_{|h|\leq \delta} \| f(\cdot +h)-f(\cdot)\|_{L^p(\R)}
\]
the modulus of continuity in the $L^p$ norm. For $p=\infty$, $\omega(f;\delta)_\infty=\omega(f;\delta)$ is the usual modulus of continuity.

We will also write $A\lesssim B$ to denote $A\leq C\cdot B$ for some constant $C$ which does not depend on essential quantities. The symbol $A\asymp B$ means that $A\lesssim B$ and $B\lesssim A$ simultaneously.

\section{Main results}

We give a couple of observations before stating our main results. First, it is easy to check by substitution that
\begin{equation}\label{adjh}
\mathcal{H}^*f(x)=\int_{\mathbb R} \varphi(t) f\Big(\frac{x}{a(t)}\Big)\,dt.
\end{equation}
Moreover, since
\[
\int_{\mathbb R}\frac{\sin\frac{s}t}s\,ds=\pi,
\]
we have
\[
\mathcal{H}^*f(x)=\frac{1}{\pi}\int_{\mathbb R} \varphi(t) \int_{\mathbb R}f\Big(\frac{x}{a(t)}\Big)\frac{\sin\frac{s}t}s\,ds\,dt.
\]
Let us now define the partial integrals
\begin{align*}
(\mathcal{H}_N \widehat{f}\,)\check{}\,(x)&=\frac1{2\pi}\int_{-N}^N \mathcal{H}\widehat{f}(u)e^{iux}\,du=\frac1{2\pi}\int_{-N}^N 
\int_{\mathbb R} \varphi(t)|a(t)|\int_{\mathbb R} f(s)e^{-ia(t)su }\,ds\,dte^{iux}\,du\\
&=\frac1{\pi}\int_{\mathbb R}\varphi(t)|a(t)|\int_{\mathbb R} f(s)\frac{\sin N(x-a(t)s)}{x-a(t)s}\,ds\,dt.
\end{align*}
By substitutions, it is easy to see that
\[
(\mathcal{H}_N \widehat{f}\,)\check{}\,(x)=\int_\R \varphi(t) \int_\R f\Big(\frac{x}{a(t)}-\frac{s}{N}\Big)\frac{\sin(a(t)s)}{s}\,ds\,dt.
\]
These observations make clear that $(\mathcal{H}_N \widehat{f}\,)\check{}$ is a natural approximation sequence and that $\mathcal{H}^*f$
is what it approximates.

Our main results read as follows.
\begin{theorem}\label{THMp<infty}
	For $1\leq p<\infty$,
	\begin{align}
	\label{EQapproximationLp}
	\frac{1}{2}\| \mathcal{H}^*f-(\mathcal{H}_N \widehat{f}\,)\check{}\|_{L^p(\R)} 
&\leq 2\int_\R |\varphi(t)||a(t)|^{1/p}\omega\Big(f;\frac{1}{|a(t)|N}\Big)_p\, dt\nonumber\\
&+\int_\R\frac{\omega(f;\frac{|s|}{N})_p }{|s|}\int_{|a^{-1}(1/s)|\leq |t|}|\varphi(t)||a(t)|^{1/p}\, dt\,ds,
	\end{align}
	where the factor $1/2$ on the left-hand side is omitted in the case $p=1$.

\end{theorem}

\begin{theorem}\label{THMp=infty}
	\begin{align}
	\label{EQcontinuousapproximation}
	\pi \| \mathcal{H}^*f-(\mathcal{H}_N \widehat{f}\,)\check{}\, \|_{L^\infty(\R)} 
&\leq 2\int_\R |\varphi(t)| \omega\Big(f;\frac{1}{N|a(t)|}\Big) \, dt\nonumber\\
 &+ \int_\R \frac{ \omega(f;\frac{|s|}{N})}{|s|}\int_{|a^{-1}(1/s)|\leq|t|} |\varphi( t)| \,dt \, ds .
	\end{align}
\end{theorem}

\begin{remark}\label{REM1}
	In order for the right-hand sides of \eqref{EQapproximationLp} and \eqref{EQcontinuousapproximation} to be finite, 
one should assume that $\varphi$ vanishes at a fast enough rate as $|t|\to \infty$, or even more, that it has compact support. 
The latter is the case for the Ces\`aro operator (where $\varphi=\chi_{(0,1)}$), which we discuss in more detail in Section~\ref{SECexamples}, along with other examples.
\end{remark}

\section{Proofs}
First of all we have the following pointwise estimate for $$ |\mathcal{H}^*f(x)-(\mathcal{H}_N \widehat{f}\,)\check{}\,(x)|,$$ 
which will be the starting point for all subsequent estimates.
\begin{lemma}\label{LEMpointwise}
	For any $x\in \R$,
	\begin{align}
	\pi |\mathcal{H}^*f(x)-(\mathcal{H}_N \widehat{f}\,)\check{}\,(x)|& \leq \int_\R |\varphi(t)a(t)|\int_{|s|\leq 1/|a(t)|} 
\Big| f\Big(\frac{x}{a(t)}-\frac{s}{N}\Big)-f\Big(\frac{x}{a(t)}\Big)\Big|\, ds\,dt \nonumber \\
	&\phantom{=}+\int_\R \frac{1}{|s|} \int_{|a^{-1}(1/s)|\leq |t|}|\varphi(t)| 
\Big| f\Big(\frac{x}{a(t)}-\frac{s}{N}\Big)-f\Big(\frac{x}{a(t)}\Big)\Big|\, dt\, ds.\label{EQbasicestimate}
	\end{align}
\end{lemma}
\begin{proof}
To prove (\ref{EQbasicestimate}), we apply rather straightforward estimates. Indeed,
\begin{align*}
\pi |\mathcal{H}^*f(x)-(\mathcal{H}_N \widehat{f}\,)\check{}\,(x)|&=\bigg| \int_\R \varphi(t) \int_\R 
\Big( f\Big(\frac{x}{a(t)}-\frac{s}{N}\Big)-f\Big(\frac{x}{a(t)}\Big)\Big) \frac{\sin (a(t)s)}{s}\, ds\, dt \bigg|\\
&\leq \bigg| \int_\R \varphi(t) \int_{|s|\leq 1/|a(t)|} \Big( f\Big(\frac{x}{a(t)}-\frac{s}{N}\Big)-f\Big(\frac{x}{a(t)}\Big)\Big) \frac{\sin (a(t)s)}{s}\, ds\, dt  \bigg|\\
&\phantom{=}+\bigg| \int_\R \varphi(t) \int_{|s|\geq 1/|a(t)|} \Big( f\Big(\frac{x}{a(t)}-\frac{s}{N}\Big)-f\Big(\frac{x}{a(t)}\Big)\Big) \frac{\sin (a(t)s)}{s}\, ds\, dt  \bigg|\\
&\leq \int_\R |\varphi(t)a(t)|\int_{|s|\leq 1/|a(t)|} \Big| f\Big(\frac{x}{a(t)}-\frac{s}{N}\Big)-f\Big(\frac{x}{a(t)}\Big)\Big|\, ds\,dt\\
&\phantom{=}+\bigg|\int_\R \frac{1}{s} \int_{|s|\leq 1/|a(t)|}\varphi(t) \Big( f\Big(\frac{x}{a(t)}-\frac{s}{N}\Big)-f\Big(\frac{x}{a(t)}\Big)\Big) \sin (a(t)s)\, dt\, ds  \bigg|\\
&\leq \int_\R |\varphi(t)a(t)|\int_{|s|\leq 1/|a(t)|} \Big| f\Big(\frac{x}{a(t)}-\frac{s}{N}\Big)-f\Big(\frac{x}{a(t)}\Big)\Big|\, ds\,dt\\
&\phantom{=}+\int_\R \frac{1}{|s|} \int_{|a^{-1}(1/s)|\leq |t|}|\varphi(t)| \Big| f\Big(\frac{x}{a(t)}-\frac{s}{N}\Big)-f\Big(\frac{x}{a(t)}\Big)\Big|\, dt\, ds,
\end{align*}
as desired. In the last inequality we use that $1/|a|$ possesses an inverse on $(0,\infty)$ (and therefore also on $(-\infty,0)$, 
since it is an odd function), and moreover $(1/|a|)^{-1}(t)=|a(1/t)|^{-1}$ on $(0,\infty)$.
\end{proof}

We now proceed to prove Theorems~\ref{THMp<infty}~and~\ref{THMp=infty}.

\begin{proof}[Proof of Theorem~\ref{THMp<infty}]
	Using \eqref{EQbasicestimate}, we get
	\begin{align*}
	&\phantom{=}\frac{1}{2}\bigg(\int_\R |\mathcal{H}^*f(x)-(\mathcal{H}_N \widehat{f}\,)\check{}\,(x)|^p\, dx\bigg)^{1/p}\\
	&\leq \bigg(\int_\R \bigg( \int_\R |\varphi(t)a(t)|\int_{|s|\leq 1/|a(t)|} \Big| f\Big(\frac{x}{a(t)}-\frac{s}{N}\Big)
-f\Big(\frac{x}{a(t)}\Big)\Big|\, ds\,dt \bigg)^p dx\bigg)^{1/p}\\
	&\phantom{=}+  \bigg(\int_\R\bigg( \int_\R \frac{1}{|s|} \int_{|a^{-1}(1/s)|\leq |t|}|\varphi(t)| 
\Big| f\Big(\frac{x}{a(t)}-\frac{s}{N}\Big)-f\Big(\frac{x}{a(t)}\Big)\Big|\, dt\, ds \bigg)^p dx\bigg)^{1/p}.
	\end{align*}
	Note that if $p=1$, the factor $\frac12$ on the left-hand side can be taken to be $1$ (in fact, such a factor appears 
due to the inequality $(a+b)^p\leq 2^p(a^p+b^p)$, for $a,b\geq 0$ and $p>1$). On the one hand, applying Minkowski's inequality twice, we get
	\begin{align*}
	&\phantom{=}  \bigg(\int_\R \bigg( \int_\R |\varphi(t)a(t)|\int_{|s|\leq 1/|a(t)|} \Big| 
f\Big(\frac{x}{a(t)}-\frac{s}{N}\Big)-f\Big(\frac{x}{a(t)}\Big)\Big|\, ds\,dt \bigg)^p dx\bigg)^{1/p}\\
	& \leq \int_\R |\varphi(t)a(t)| \bigg( \int_\R \bigg(\int_{|s|\leq 1/|a(t)|} \Big| 
f\Big(\frac{x}{a(t)}-\frac{s}{N}\Big)-f\Big(\frac{x}{a(t)}\Big)\Big|\, ds\bigg)^p dx \bigg)^{1/p} dt\\
	&\leq \int_\R |\varphi(t)a(t)| \int_{|s|\leq 1/|a(t)|}\bigg( \int_\R  \Big| 
f\Big(\frac{x}{a(t)}-\frac{s}{N}\Big)-f\Big(\frac{x}{a(t)}\Big)\Big|^p\, dx\bigg)^{1/p} ds \, dt\\
	&=\int_\R |\varphi(t)||a(t)|^{1+1/p} \int_{|s|\leq 1/|a(t)|}\bigg( \int_\R  \Big| f\Big(x-\frac{s}{N}\Big)-f(x)\Big|^p\, dx\bigg)^{1/p} ds \, dt\\
	&\leq \int_\R |\varphi(t)||a(t)|^{1+1/p} \int_{|s|\leq 1/|a(t)|} \omega\Big(f; \frac{|s|}{N}\Big)_p\,ds\,dt.
	\end{align*}
	Since $\omega(f;\delta)_p$ is nondecreasing in $\delta$, we have
	\[
	\int_\R |\varphi(t)||a(t)|^{1+1/p}  \int_{|s|\leq 1/|a(t)|} \omega\Big(f;\frac{|s|}{N}\Big)_p\, ds\, dt\leq 2\int_\R |\varphi(t)||a(t)|^{1/p}\omega\Big(f;\frac{1}{|a(t)|N}\Big)_p\, dt.
	\]
	On the other hand, applying Minkowski's inequality again, we obtain
	\begin{align*}
	&\phantom{=} \bigg(\int_\R\bigg( \int_\R \frac{1}{|s|} \int_{|a^{-1}(1/s)|\leq |t|}|\varphi(t)| \Big| 
f\Big(\frac{x}{a(t)}-\frac{s}{N}\Big)-f\Big(\frac{x}{a(t)}\Big)\Big|\, dt\, ds \bigg)^p dx\bigg)^{1/p}\\
	&\leq \int_\R\frac{1}{|s|} \bigg(\int_\R  \bigg(\int_{|a^{-1}(1/s)|\leq |t|}|\varphi(t)| \Big| 
f\Big(\frac{x}{a(t)}-\frac{s}{N}\Big)-f\Big(\frac{x}{a(t)}\Big)\Big|\, dt \bigg)^p dx\bigg)^{1/p} ds \\
	&\leq \int_\R\frac{1}{|s|} \int_{|a^{-1}(1/s)|\leq |t|}|\varphi(t)| \bigg( \int_\R  \Big| 
f\Big(\frac{x}{a(t)}-\frac{s}{N}\Big)-f\Big(\frac{x}{a(t)}\Big)\Big|^p\, dx \bigg)^{1/p} dt\, ds \\
	&= \int_\R\frac{1}{|s|} \int_{|a^{-1}(1/s)|\leq |t|}|\varphi(t)||a(t)|^{1/p} \bigg( \int_\R  \Big| f\Big(x-\frac{s}{N}\Big)-f(x)\Big|^p\, dx \bigg)^{1/p} dt\, ds\\
	&\leq \int_\R\frac{\omega(f;\frac{|s|}{N})_p }{|s|}\int_{|a^{-1}(1/s)|\leq |t|}|\varphi(t)||a(t)|^{1/p}\, dt\,ds.
	\end{align*}
	Collecting all the estimates, we derive
	\begin{align*}
	\frac{1}{2}\| \mathcal{H}^*f-(\mathcal{H}_N \widehat{f}\,)\check{}\|_{L^p(\R)} &\leq 2\int_\R |\varphi(t)||a(t)|^{1/p}\omega\Big(f;\frac{1}{|a(t)|N}\Big)_p\, dt\\
&+\int_\R\frac{\omega(f;\frac{|s|}{N})_p }{|s|}\int_{|a^{-1}(1/s)|\leq |t|}|\varphi(t)||a(t)|^{1/p}\, dt\,ds,
	\end{align*}
	where the factor $1/2$ on the left-hand side is omitted in the case $p=1$. The proof is complete.
\end{proof}

\begin{proof}[Proof of Theorem~\ref{THMp=infty}]

It suffices to estimate the two terms on the right-hand side of \eqref{EQbasicestimate} in the $L^\infty$ norm. For the first one, we have
\begin{align*}
\int_\R |\varphi(t)a(t)|\int_{|s|\leq 1/|a(t)|} \Big| f\Big(\frac{x}{a(t)}-\frac{s}{N}\Big)-f\Big(\frac{x}{a(t)}\Big)\Big|\, ds\,dt 
\leq \int_\R |\varphi(t)a(t)|   \int_{|s|\leq 1/|a(t)|}\omega\Big(f;\frac{|s|}{N}\Big)\, ds\,dt,
\end{align*}
and since $\omega(f;\delta)$ is nondecreasing in $\delta$, we obtain
\begin{equation}
\label{EQterm1}
 \int_\R |\varphi(t)a(t)| \int_{|s|\leq 1/|a(t)|}\omega\Big(f;\frac{|s|}{N}\Big)\, ds\,dt\leq 2\int_\R |\varphi(t)| \omega\Big(f;\frac{1}{N|a(t)|}\Big)\, dt,
\end{equation}

As for the second term,
\begin{align*}
&\int_\R \frac{1}{|s|} \int_{|a^{-1}(1/s)|\leq |t|}|\varphi(t)| \Big| f\Big(\frac{x}{a(t)}-\frac{s}{N}\Big)-f\Big(\frac{x}{a(t)}\Big)\Big|\, dt\, ds\\
 \leq &\int_\R \frac{1}{|s|} \int_{|a^{-1}(1/s)|\leq |t|} |\varphi(t)| \omega\Big(f;\frac{|s|}{N}\Big)\, dt\,ds \\
= &\int_\R \frac{ \omega(f;\frac{|s|}{N})}{|s|} \int_{|a^{-1}(1/s)|\leq |t|} |\varphi(t)|\,dt\,ds.
\end{align*}
Collecting all the estimates, we get
\begin{align*}
\pi |\mathcal{H}^*f(x)-(\mathcal{H}_N \widehat{f}\,)\check{}\,(x)| &\leq 2\int_\R |\varphi(t)| \omega\Big(f;\frac{1}{N|a(t)|}\Big)\, dt\\
 &+ \int_\R \frac{ \omega(f;\frac{|s|}{N})}{|s|}\int_{|a^{-1}(1/s)|\leq|t|} |\varphi( t)| \,dt \, ds,
\end{align*}
where the right-hand side is uniform in $x$.
\end{proof}

\section{Examples}\label{SECexamples}

Let us consider some examples of Hausdorff operators by choosing particular functions $\varphi$ and $a$. 
We shall give bounds for the approximation error explicitly in $L^p$, $1\leq p\leq \infty$, in each case.

\subsection{Approximating the Ces\`aro operator}

The Ces\`aro operator $\mathcal{C}$ is the prototype Hausdorff operator $\mathcal{H}_{\varphi,a}$ \cite{emj,LM1}. Letting $a(t)=1/t$ and $\varphi(t)=\chi_{(0,1)}(t)$, we obtain
\[
\mathcal{C}^*f(x)=\int_0^1  f(tx)\,dt=\frac{1}{x}\int_0^x f(t)\,dt.
\]
This is the Ces\`aro (Hardy) operator.
We have
\[
(\mathcal{C}_N \widehat{f}\,)\check{}\,(x):=(\mathcal{H}_N \widehat{f}\,)\check{}\,(x) =\frac1{\pi}\int_0^1 \frac{1}{t}\int_{\mathbb R} f(s)\frac{\sin N(x-s/t)}{x-s/t}\,ds\,dt.
\]
In this case, Theorem~\ref{THMp<infty} yields, for $1<p<\infty$,
\begin{align*}
\| \mathcal{C}^*f-(\mathcal{C}_N \widehat{f}\,)\check{}\|_{L^p(\R)} &\leq 4\int_0^1 \frac{ \omega(f;\frac{t}{N})_p}{t^{1/p}}\, dt
+2\int_\R\frac{\omega(f;\frac{|s|}{N})_p }{|s|}\int_{|s|\leq  |t|}|\varphi(t)||a(t)|^{1/p}\, dt\,ds \\
&=4\int_0^1 \frac{ \omega(f;\frac{t}{N})_p}{t^{1/p}}\, dt+2\int_0^1 \frac{\omega(f;\frac{s}{N})_p }{s}
\int_{s}^1 \frac{1}{t^{1/p}}\, dt\,ds\asymp \int_0^1 \frac{ \omega(f;\frac{t}{N})_p}{t}\, dt.
\end{align*}
The case $p=1$ is different, more precisely, we obtain an additional logarithm as compared with the case $1<p<\infty$. Indeed, Theorem~\ref{THMp<infty} yields
\[
\| \mathcal{C}^*f-(\mathcal{C}_N \widehat{f}\,)\check{}\|_{L^1(\R)} \leq  2\int_0^1\frac{\omega(f;\frac{t}{N})_1}{t}\,
+2\int_0^1\frac{\omega(f;\frac{s}{N})_1}{s}\int_s^1\frac1t \, dt\,ds \asymp \int_0^1 \omega\Big(f;\frac{t}{N}\Big)_1\frac{|\log t|}{t}\,dt,
\]
Finally, in the case $p=\infty$ we obtain a Dini-type condition from Theorem~\ref{THMp=infty},
\[
\pi \| \mathcal{C}^*f-(\mathcal{C}_N \widehat{f}\,)\check{}\|_{L^\infty(\R)}	 \leq 2\int_0^1 \omega\Big(f;\frac{t}{N}\Big) \, dt 
+2 \int_0^1 \frac{ \omega(f;\frac{|s|}{N})}{|s|}(1-s) \, ds \asymp \int_0^1 \frac{\omega(f;\frac{t}{N})}{t} \, dt .
\]
Now, given a continuous function $f$, if
\[
\int_0^1\frac{\omega(f;t)}{t}<\infty,
\]
then
\[
\int_0^1\frac{\omega(f;\frac{t}{N})}{t}\to 0\qquad \text{as }N\to\infty.
\]
Therefore, we can conclude the following.
\begin{corollary}\label{CORcesaro}
	Let $f$ be a function defined on $\R$.
	\begin{enumerate}
	\item If $f\in L^p(\R)$, $1< p<\infty$, then
	\[
	\| \mathcal{C}^*f-(\mathcal{C}_N \widehat{f}\,)\check{}\|_{L^p(\R)} \lesssim \int_0^1 \frac{ \omega(f;\frac{t}{N})_p}{t}\, dt.
	\]
	In particular, if $\int_0^1 \frac{ \omega(f;t)_p}{t}\, dt<\infty$, then $(\mathcal{C}_N \widehat{f}\,)\check{}$ converges to $\mathcal{C}^*f$ in $L^p(\R)$.
	\item If $f\in L^1(\R)$, then
	\[
	\| \mathcal{C}^*f-(\mathcal{C}_N \widehat{f}\,)\check{}\|_{L^1(\R)} \lesssim \int_0^1 \omega\Big(f;\frac{t}{N}\Big)_1\frac{|\log t|}{t}\,dt.
	\]
	In particular, if $\int_0^1 \omega(f;t)_1\frac{|\log t|}{t}\,dt<\infty$, then  $(\mathcal{C}_N \widehat{f}\,)\check{}$ converges to $\mathcal{C}^*f$ in $L^1(\R)$.
	\item For the uniform norm, we have
	\[
	\| \mathcal{C}^*f-(\mathcal{C}_N \widehat{f}\,)\check{}\|_{L^\infty(\R)}\lesssim \int_0^1 \frac{\omega(f;\frac{t}{N})}{t} \, dt .
	\]
	In particular, if $f$ is continuous and $\int_0^1\frac{\omega(f;t)}{t}<\infty$, then $(\mathcal{C}_N \widehat{f}\,)\check{}$ converges uniformly to $\mathcal{C}^*f$ on $\R$.
	\end{enumerate}
\end{corollary}

We shall now compare the approximation estimates from Corollary~\ref{CORcesaro} with those for approximate identities.

\subsection{Comparison: Ces\`aro operators and approximate identities}
Since the Ces\`aro operator is the prototype example of Hausdorff operator, we are also interested in comparing the 
obtained approximations with the classical ones given by approximate identities  for convolutions.
A family of functions $\{F_r\}_{r>0}$ defined on $\R$ is called an \textit{approximate identity} if
\begin{enumerate}
	\item $\sup_r\| F_r\|_{L^1(\R)}<\infty$, and
	\item for every $\delta>0$,
	\[
	\int_{|x|\geq \delta} |F_r(x)|\,dx\to 0\qquad \text{as }r\to\infty.
	\]
\end{enumerate}
The following is well known \cite[Theorem 3.1.6]{BN}.
\begin{thmletter}\label{THMapproxidentity}
	Let $g\in L^p(\R)$, with $1\leq p <\infty$. If $\{F_r\}_{r>0}$ is an approximate identity satisfying
	\begin{equation}
	\label{EQintFr}
	\int_\R F_r(x)\,dx=1,\qquad r>0,
	\end{equation}
	then
	\[
	\|F_r*g-g\|_{L^p(\R)}\to 0,\qquad \text{as }r\to \infty.
	\]
\end{thmletter}
As an example of an approximate identity satisfying \eqref{EQintFr}, we have the family of functions
\[
F_r(x)= rF(rx),\qquad r>0,
\]
where $F(x)$ is the \textit{Fej\'er kernel} on the real line,
\[
F(x)=\frac{1}{2\pi}\bigg( \frac{\sin(x/2)}{x/2}\bigg)^2.
\]
From now on we assume that the approximate identities we consider satisfy condition \eqref{EQintFr}.

Comparing Theorem~\ref{THMapproxidentity}~and~Corollary~\ref{THMp<infty}, we readily see that the latter requires 
further assumptions in order to guarantee $L^p$ convergence $(p<\infty)$, namely that the integral 
\[
\int_0^1 \frac{ \omega(f;t)_p}{t}\, dt
\]
is finite for $p>1$, and that 
\[
\int_0^1 \omega(f;t)_1\frac{|\log t|}{t}\,dt<\infty
\] 
is finite for $p=1$. However, when restricted to certain classes of functions, the approximation rates become the same, or even better.

As classes of functions, we consider $\Lambda^p_\alpha(\R)$ with $0<\alpha\leq 1$, and $1\leq p\leq \infty$, which consists of the functions $f$ satisfying
\[
\omega(f;\delta)_p \lesssim \delta^{\alpha},\qquad \delta>0.
\]
Note that $\Lambda^\infty_\alpha(\R)$ is the class of  usual Lipschitz $\alpha$ continuous functions on $\R$, i.e., those satisfying
\[
|f(x)-f(y)|\leq C|x-y|^\alpha, \qquad x,y\in \R.
\]
For $f\in \Lambda^p_\alpha(\R)$, $0<\alpha<1$, and $1\leq p<\infty$, it is known that any approximate identity $\{F_r\}$ yields the approximation rate
\begin{equation}
\label{EQapproxlipa}
\|  F_r*g-g\|_{L^p(\R)} \lesssim \frac{1}{r^\alpha}, \qquad r\to \infty,
\end{equation}
see \cite[Corollary~3.4.4]{BN}, whilst for $\alpha=1$, an additional logarithm appears,
\begin{equation}
\label{EQapproxlip1}
\|  F_r*g-g\|_{L^p(\R)} \lesssim \frac{\log r}{r}, \qquad r\to \infty,
\end{equation}
cf. \cite[Problem 3.4.2]{BN}. In the case of the Ces\`aro operator, Corollary~\ref{CORcesaro} yields, for any $1<p<\infty$ and $f\in \Lambda^p_\alpha(\R)$
\[
\| \mathcal{C}^*f-(\mathcal{C}_N \widehat{f}\,)\check{}\|_{L^p(\R)} \lesssim \int_0^1 \frac{ \omega(f;\frac{t}{N})_p}{t}\, dt
\lesssim N^{-\alpha} \int_0^1 t^{\alpha-1}\,dt\asymp N^{-\alpha},\qquad N\to\infty,
\]
and
\[
\| \mathcal{C}^*f-(\mathcal{C}_N \widehat{f}\,)\check{}\|_{L^1(\R)} \lesssim \int_0^1 \omega\Big(f;\frac{t}{N}\Big)_1\frac{|\log t|}{t}\,dt 
\lesssim N^{-\alpha}\int_0^1 t^{\alpha-1}|\log t|\,dt\asymp N^{-\alpha},\qquad N\to \infty.
\]
These approximation rates are the same as those for approximate identities when restricted to functions 
$f\in \Lambda^p_\alpha(\R)$ with $0<\alpha<1$ (compare with \eqref{EQapproxlipa}), and are actually better than their counterpart in 
the case $\alpha=1$ (compare with \eqref{EQapproxlip1}), in the sense that the extra logarithm from \eqref{EQapproxlip1} does not appear.

As for the case $p=\infty$, we have the following estimates for approximate identities. Let $f$ be a bounded function. If $f\in \Lambda^\infty_\alpha(\R)$ with $0<\alpha<1$, then \cite[Corollary~3.4.4]{BN}
\[
\| F_r*g-g\|_{L^\infty (\R)} \lesssim \frac{1}{r^\alpha},\qquad r\to\infty,
\]
whilst if $f\in \Lambda^\infty_1(\R)$, then, similarly as in \eqref{EQapproxlip1},
\[
\|F_r*g-g\|_{L^\infty (\R)}  \lesssim \frac{\log r}{r}, \qquad r\to \infty,
\]
see \cite[Problem 3.4.2]{BN}. Corollary~\ref{CORcesaro} yields the following rates of approximation 
for Ces\`aro operators for $f\in \Lambda^\infty_\alpha(\R)$ with $0<\alpha\leq 1$,
\[
\| \mathcal{C}^*f-(\mathcal{C}_N \widehat{f}\,)\check{}\|_{L^\infty(\R)}\lesssim \int_0^1 \frac{\omega(f;\frac{t}{N})}{t} \, dt 
\lesssim N^{-\alpha} \int_0^1 t^{\alpha-1}\,dt\asymp N^{-\alpha},\qquad N\to\infty.
\]
Again we obtain the same rates of approximation for functions $f\in \Lambda^\infty_\alpha(\R)$ as those for approximate identities, 
whenever $0<\alpha<1$. On the other hand, in the case $f\in\Lambda^\infty_1(\R)$, the ``Hausdorff" approximation rate is 
better than \eqref{EQapproxlip1}, differing by a logarithm.

\subsection{Approximating the adjoint Ces\`aro operator}

We now wish to approximate the adjoint Ces\`aro operator (Bellman operator) $\mathcal{B}$. Its adjoint $\mathcal{B}^*$ is defined by letting $a(t)=1/t$ and 
$\varphi(t)=t^{-1}\chi_{(1,\infty)}(t)$ in (\ref{adjh}):
\[
\mathcal{B}^* f(x)=\int_1^\infty \frac{f(tx)}{t}\,dt=\frac{1}{x}\int_x^\infty \frac{f(t)}{t}\,dt.
\]
It will be approximated by means of
\[
(\mathcal{B}_N \widehat{f}\,)\check{}\,(x):=\frac1{\pi}\int_1^\infty \frac{1}{t^2}\int_{\mathbb R} f(s)\frac{\sin N(x-s/t)}{x-s/t}\,ds\,dt.
\]
For $1\leq p<\infty$, one has, by Theorem~\ref{THMp<infty},
\begin{align}
\| \mathcal{B}^*f-(\mathcal{B}_N \widehat{f}\,)\check{}\|_{L^p(\R)}&\leq  2\int_1^\infty \frac{\omega(f;\frac{t}{N})_p}{t^{1+1/p}}\, dt
+\int_\R\frac{\omega(f;\frac{|s|}{N})_p }{|s|}\int_{|s|\leq |t|}\frac{1}{t^{1+1/p}}\chi_{(1,\infty)}(t)\, dt\,ds\nonumber\\
&\asymp \int_0^\infty \frac{\omega(f;\frac{|s|}{N})_p }{|s|}\,ds, \label{EQdualapprox}
\end{align}
whilst in the case $p=\infty$, Theorem~\ref{THMp=infty} yields
\[
\pi \| \mathcal{B}^*f-(\mathcal{B}_N \widehat{f}\,)\check{}\|_{L^\infty(\R)}\leq 2\int_1^\infty  
\frac{\omega(f;\frac{t}{N})}{t} \, dt + \int_\R \frac{ \omega(f;\frac{|s|}{N})}{|s|}\int_{|s|\leq|t|} \frac{1}{t}\chi_{(1,\infty)}(t) \,dt \, ds=\infty,
\]
i.e., in this case we cannot guarantee any convergence on the $L^\infty$ norm by using our estimates, even for well-behaved functions $f$. 
As was pointed out in Remark~\ref{REM1}, this is because in order to obtain useful estimates from Theorems~\ref{THMp<infty}~and~\ref{THMp=infty}, 
one should assume that $\varphi$ is of compact support, or that decays fast enough as $|t|\to \infty$. For the adjoint Ces\`aro operator the functions 
$\varphi$ has some decay, but it is certainly not enough. The estimate \eqref{EQdualapprox} is also not very promising.

\subsection{The Riemann-Liouville integral}
For $\alpha>0$, the Riemann-Liouville integral is defined as
\[
\mathscr{I}^\alpha f(x)= \frac{1}{\Gamma(\alpha)} \int_0^x f(t) (x-t)^\alpha\,dt =\frac{x^\alpha}{\Gamma(\alpha)}\int_0^x f(t) \Big(1-\frac{t}{x}\Big)^\alpha\,dt.
\]
A re-scaled version of this operator may be easily obtained as a dual Hausdorff operator. 
Indeed, for $a(t)=1/t$ and $\varphi(t)=(1-t)^\alpha \chi_{(0,1)}(t)$, we have
\[
\mathcal{I}^\alpha(f)(x):= \mathcal{H}^*f(x)= \int_0^1 f(tx) (1-t)^\alpha \,dt = \frac{1}{x} 
\int_0^1 f(t)\Big(1-\frac{t}{x} \Big)^\alpha\,dt = \Gamma (\alpha) x^{-\alpha-1}\mathscr{I}^\alpha f(x).
\]
Using Theorems~\ref{THMp<infty}~and~\ref{THMp=infty} we approximate $\mathcal{I}^\alpha(f)$ by
\[
(\mathcal{I}^\alpha_N \widehat{f}\,)\check{}\,(x):=(\mathcal{H}_N \widehat{f}\,)\check{}\,(x)=\frac1{\pi}\int_{0}^1 
\frac{(1-t)^\alpha}{t} \int_{\mathbb R} f(s)\frac{\sin N(x-s/t)}{x-s/t}\,ds\,dt.
\]
In $L^p$, for $1<p<\infty$, we have
\begin{align*}
\| \mathcal{I}^\alpha f-(\mathcal{I}^\alpha_N \widehat{f}\,)\check{}\|_{L^p(\R)}&\leq 4\int_0^1 \frac{(1-t)^\alpha}{t^{1/p}} 
\omega\Big(f;\frac{t}{N}\Big)_p\, dt+2\int_\R\frac{\omega(f;\frac{|s|}{N})_p}{|s|}\int_{|s|\leq |t|}\frac{(1-t)^\alpha}{t^{1/p}}\chi_{(0,1)}(t)\, dt\,ds\\
&\asymp \int_0^1 \frac{(1-t)^\alpha}{t^{1/p}} \omega\Big(f;\frac{t}{N}\Big)_p\,dt+ \int_0^1\frac{\omega(f;\frac{s}{N})_p }{s}
\int_{s}^1\frac{(1-t)^\alpha}{t^{1/p}}\, dt\,ds\asymp \int_0^1\frac{\omega(f;\frac{t}{N})_p }{t}\,dt.
\end{align*}
Secondly, for $L^1$, we obtain
\begin{align*}
\| \mathcal{I}^\alpha f-(\mathcal{I}^\alpha_N \widehat{f}\,)\check{}\|_{L^1(\R)}& \leq 2\int_0^1 \frac{(1-t)^\alpha}{t} 
\omega\Big(f;\frac{t}{N}\Big)_1\, dt+2\int_0^1\frac{\omega(f;\frac{s}{N})_1 }{s}\int_{s}^1 \frac{(1-t)^\alpha}{t}\, dt\,ds,\\
&\asymp \int_0^1 \frac{|\log t|}{t}\omega\Big(f;\frac{t}{N}\Big)_1\, dt,
\end{align*}
and finally, in $L^\infty$ we have the estimate
\[
\pi \| \mathcal{I}^\alpha f-(\mathcal{I}^\alpha_N \widehat{f}\,)\check{}\|_{L^\infty(\R)} \leq 2\int_0^1 (1-t)^\alpha  
\omega\Big(f;\frac{t}{N}\Big) \, dt +2 \int_0^1 \frac{ \omega(f;\frac{s}{N})}{s}
\int_{s}^1 (1-t)^\alpha \,dt \, ds\asymp \int_0^1 \frac{ \omega(f;\frac{s}{N})}{s}\, ds.
\]
Note that these estimates are the same as those for the Ces\`aro operator. More generally, one can readily see from the estimates 
in Theorems~\ref{THMp<infty}~and~\ref{THMp=infty} that, fixing $a(t)=1/t$, any Hausdorff operator with $\varphi$ supported on $(0,1)$ such that  
$\varphi(t)\asymp 1$ in some neighbourhood $(0,\varepsilon)$ will satisfy such approximation estimates. In fact, if $\varphi$ is supported 
on $(0,r)$ with $r>0$ rather than on $(0,1)$, these estimates are still valid, with the integration carried out over $(0,r)$.

\section{Final remarks}\label{SEC5}
We conclude with a couple of remarks: first, we show that one may use the same approach to approximate the Hausdroff operator (instead of its adjoint) applied to a function. Secondly, we give an application of Theorem~\ref{THMp=infty} to approximate functions in the $L^\infty$ norm.
\subsection{Approximation of non-adjoint Hausdorff operators}
One can also approximate the Hausdorff operator instead of its adjoint, if one considers the adjoint Hausdorff averages in the approximant. More precisely, it is also possible to approximate $\mathcal{H}f(x)$ by
\begin{align*}
(\mathcal{H}^*_N \widehat{f})\check{}(x)&=\frac1{2\pi}\int_{-N}^N \mathcal{H}^*\widehat{f}(u)e^{iux}\,du=\frac1{2\pi}\int_{-N}^N\int_\R\varphi(t) \int_\R f(s) e^{-isu/a(t)}ds\,dt\,e^{iux}du\\
&=\frac1\pi \int_\R \varphi(t)\int_\R f(s)\frac{\sin N(x-s/a(t))}{x-s/a(t)}\,ds\,dt
\end{align*}
which, by substitutions, is easily seen to equal
\[
\frac1\pi \int_\R \varphi(t)|a(t)| \int_\R f\Big( a(t)x -\frac{s}{N}\Big) \frac{\sin \frac{s}{a(t)}}{s}\,ds\,dt.
\]
Since for any $t\neq 0$ one has
\[
\int_{\R}\frac{\sin\frac{s}{a(t)}}{s}\,du=\pi,
\]
then
\[
\mathcal{H}f(x)-(\mathcal{H}^*_N \widehat{f})\check{}(x) =\frac{1}{\pi} \int_\R \varphi(t)|a(t)| \int_\R \Big( f\Big(a(t) x-\frac{s}{N} \Big)-f\big(a(t)x\big) \Big)\frac{\sin \frac{s}{a(t)}}{s}\,ds\,dt.
\]
A similar estimate to  that of Lemma~\ref{LEMpointwise} can now be proved.
\begin{lemma}\label{LEMpointwise2}
	For any $x\in \R$,
	\begin{align*}
	\pi | \mathcal{H}f(x)-(\mathcal{H}^*_N \widehat{f})\check{}(x)|&\leq  \int_{\R} \varphi(t) \int_{|s|\leq |a(t)|} \Big| f\Big(a(t) x-\frac{s}{N} \Big)-f\big(a(t)x\big) \Big|\,ds\,dt \\
	& \phantom{=} +  \int_{\R} \frac1s\int_{|t|\geq |a^{-1}(s)|} \Big| f\Big(a(t) x-\frac{s}{N} \Big)-f\big(a(t)x\big) \Big| \varphi(t)|a(t)|\, dt\, ds,
	\end{align*}
\end{lemma}
\begin{proof} 
	The proof is essentially the same as that of Lemma~\ref{LEMpointwise},
	\begin{align*}
	\pi | \mathcal{H}f(x)-(\mathcal{H}^*_N \widehat{f})\check{}(x)| &= \int_{\R} \varphi(t)|a(t)| \int_{|s|\leq |a(t)|} \Big( f\Big(a(t) x-\frac{s}{N} \Big)-f\big(a(t)x\big) \Big)\frac{\sin \frac{s}{a(t)}}{s}\,ds\,dt\\
	&\phantom{=}+\int_{\R} \varphi(t)|a(t)| \int_{|s|\geq |a(t)|} \Big( f\Big(a(t) x-\frac{s}{N} \Big)-f\big(a(t)x\big) \Big)\frac{\sin \frac{s}{a(t)}}{s}\,ds\,dt\\
	&\leq \int_{\R} \varphi(t) \int_{|s|\leq |a(t)|} \Big| f\Big(a(t) x-\frac{s}{N} \Big)-f\big(a(t)x\big) \Big|\,ds\,dt\\
	&\phantom{=}+ \int_{\R} \varphi(t)|a(t)| \int_{|s|\geq |a(t)|} \frac{ \Big| f\Big(a(t) x-\frac{s}{N} \Big)-f\big(a(t)x\big) \Big|}{|s|}\,ds\,dt,\\
	&\leq  \int_{\R} \varphi(t) \int_{|s|\leq |a(t)|} \Big| f\Big(a(t) x-\frac{s}{N} \Big)-f\big(a(t)x\big) \Big|\,ds\,dt \\
	& \phantom{=} +  \int_{\R} \frac1s\int_{|t|\geq |a^{-1}(s)|} \Big| f\Big(a(t) x-\frac{s}{N} \Big)-f\big(a(t)x\big) \Big| \varphi(t)|a(t)|\, dt\, ds,
	\end{align*}
	as desired.
\end{proof}
By means of the pointwise estimate from Lemma~\ref{LEMpointwise2} it is possible to obtain approximation results analogous to Theorems~\ref{THMp<infty}~and~\ref{THMp=infty}, where the Hausdorff operator, rather than its adjoint, is approximated. The details are essentially the same and are thus omitted.
\subsection{Approximation of functions}
The fact that either the Hausdorff operator of a function or is adjoint is approximated is unsatisfactory in principle, as one would like to approximate the function itself. However, approximating a function instead of its (adjoint) Hausdorff operator is also possible, under the following observation. For $\varphi\in L^1(\R)$ and $a(t)$ as in the Introduction, one has
\[
\mathcal{H}^*_{\varphi,a}f (0) =\int_\R \varphi(t) f\Big( \frac{0}{a(t)}\Big)\,dt=f(0)\int_\R \varphi(t)\,dt.
\]
If we denote by $\tau_yf(x)=f(x+y)$ the translation of $f$ by $y\in \R$, and assume that $\int_\R \varphi(t)\,dt=1$, then
\[
\mathcal{H}^*_{\varphi,a}  [\tau_y f](0)=f(y).
\]
This gives a natural way of approximating $f$ through Hausdorff operators, namely, if we define the function $F_N(y)= (\mathcal{H}_N \widehat{ \tau_yf}\,)\check{}\,(0)=(\mathcal{H}_N[e^{iy\cdot}\widehat{f}(\cdot)]\,)\check{}\,(0)$, $y\in \R$, then in particular Theorem~\ref{THMp=infty} asserts that
\[
\pi \| f(y)-F_N(y)\|_{L^\infty(\R)}\leq 2\int_\R |\varphi(t)| \omega\Big(f;\frac{1}{N|a(t)|}\Big) \, dt\nonumber\\
+ \int_\R \frac{ \omega(f;\frac{|s|}{N})}{|s|}\int_{|a^{-1}(1/s)|\leq|t|} |\varphi( t)| \,dt \, ds ,
\]
since $\omega(f;\delta)=\omega(\tau_yf;\delta)$ for any $y\in \R$ and $\delta>0$.
%
%


\begin{thebibliography}{99}

\bibitem{AL} L. Aizenberg, E. Liflyand, {\it Hardy spaces in Reinhardt domains,	and Hausdorff operators}, Illinois J. Math. {\bf 53} (2009), 1033--1049.
	
\bibitem{BM}
G. Brown and F. M\'oricz,
\newblock \textit{Multivariate {H}ausdorff operators on the spaces {$L^p(\mathbb{R}^n)$}},
\newblock J. Math. Anal. Appl. \textbf{271} (2) (2002), 443--454.

\bibitem{BN}
P. L. Butzer and R. J. Nessel,
\newblock \textit{Fourier Analysis and Approximation},
\newblock Academic Press, New York-London, 1971.

\bibitem{CFW} J. Chen, D. Fan and S. Wang, {\it Hausdorff Operators on Euclidean Spaces}, Appl. Math. J. Chinese Univ. (Ser. B) {\bf 28}(4) (2014), 548--564.

\bibitem{K} Y.\ Kanjin, {\it The Hausdorff operators on the real Hardy spaces $H^{p}(\mathbb R)$}, Studia Math.\ \textbf{148} (2001), 37--45.

\bibitem{Kuang} J.C. Kuang, {\em Generalized Hausdorff operators on
weighted Morrey-Herz spaces}\/, Acta Math. Sinica (Chin. Ser.)  {\bf 55}
(2012), 895--902 (Chinese; Chinese, English summaries).

\bibitem{Kuang1} J.C. Kuang, {\it Generalized Hausdorff operators on weighted Herz spaces}, Mat. Vesnik {\bf 66} (2014), 19--32.

\bibitem{Ge} C. Georgakis, {\it The Hausdorff mean of a Fourier-Stieltjes
transform}, Proc. Am. Math. Soc. {\bf 116} (1992), 465--471.

\bibitem{LL} A. Lerner and E. Liflyand, {\it Multidimensional Hausdorff operator on the real Hardy space}, J. Austr. Math. Soc. {\bf 83} (2007), 79--86.

\bibitem{emj} E. Liflyand, {\it Hausdorff Operators on Hardy Spaces}, Eurasian Math. J. {\bf 4} (4), 101--141.

\bibitem{Lif} E. Liflyand, {\it Open problems on Hausdorff operators}, In: Complex Analysis and Potential Theory, Proc. Conf. Satellite to ICM 2006, Gebze, Turkey, 8-14 Sept. 2006; Eds. T. Aliyev Azeroglu and P.M. Tamrazov; World Sci., 2007, 280--285.

\bibitem{LiMi}
E. Liflyand and A. Miyachi, {\it Boundedness of the Hausdorff operators	in $H^p$ spaces, $0<p<1$}, Studia Math. {\bf 194}(3) (2009), 279--292.

\bibitem{LiMi1}
E.\ Liflyand and A.\ Miyachi, {\it Boundedness of multidimensional Hausdorff operators in $H^p$ spaces, $0<p<1$}, Trans. Amer. Math. Soc. {\bf 371} (2019), 4793--4814.

\bibitem{LM1} E. Liflyand and F. M\'oricz, {\it The Hausdorff operator is bounded on the real Hardy space
$H^{1} ({\mathbb{R}})$}, Proc. Am. Math. Soc. {\bf 128} (2000), 1391--1396.

\bibitem{Mir} A. R. Mirotin, {\it Boundedness of Hausdorff operators on real Hardy spaces $H^1$ over locally compact groups}, J. Math. Anal. Appl. {\bf 473} (2019), 519--533.
\end{thebibliography}
\end{document}